\documentclass[11pt,twoside]{amsart}
\usepackage{amssymb}
\usepackage{graphicx}
\usepackage{float}

\newtheorem{theorem}{Theorem}[section]
\newtheorem{lemma}[theorem]{Lemma}
\newtheorem{proposition}[theorem]{Proposition}

\newtheorem{corollary}[theorem]{Corollary}

\theoremstyle{definition}

\theoremstyle{remark}
\newtheorem{remark}[theorem]{Remark}

\def\supp{{\mathrm{supp}\ }}

\def\C{\mathbb C}

\def\N{\mathbb N}

\def\Lt{{\widetilde{L}}}

\def\Ht{{\widetilde{H}}}

\def\eps{\varepsilon}
\def\3{{\ss}}

\def\supp{{\mathrm{supp}\ }}

\def\sing{{\mathrm{sing\,}}}
\def\reg{{\mathrm{reg\,}}}

\def\eta{{(e^{-tA})_{t\geq 0}}}

\newcommand{\norm}[1]{\left\Vert#1\right\Vert}

\newcommand{\be}{\begin{equation}}
\newcommand{\ee}{\end{equation}}

\newcommand{\bea}{\begin{eqnarray}}
\newcommand{\eea}{\end{eqnarray}}
\newcommand{\ben}{\begin{eqnarray*}}
\newcommand{\een}{\end{eqnarray*}}

\numberwithin{equation}{section}

\dedicatory{Dedicated to Fritz Gesztesy on the occasion of his   sixtieth  birthday.}
\thanks{The authors would like to thank the referee for carefully reading an earlier version of the manuscript and Dr.~Rob Davies for producing the figures.}

\title[Rooms and Passages domains]{Some spectral properties of Rooms and Passages domains and their skeletons}
\author{B. M. Brown}
\address{Cardiff School of Computer Science and Informatics, Cardiff University,  Queen's Buildings, 5 The Parade, Roath, Cardiff CF24 3AA}
\email{malcolm@cs.cf.ac.uk}

\author{W. D. Evans}
\address{School of Mathematics Cardiff University,
Senghennydd Road, 
Cardiff, Wales, UK.
CF24 4AG}
\email{EvansWD@cardiff.ac.uk}

\author{I. G. Wood}
\address{School of Mathematics, Statistics  and  Actuarial Science,
Cornwallis Building, 
University of Kent,
Canterbury,
Kent, CT2 7NF}
\email{i.wood@kent.ac.uk}

\date{}
\begin{document}
\begin{abstract}

In this paper we investigate spectral properties of Laplacians on Rooms and Passages domains. In the first part, we use Dirichlet-Neumann bracketing techniques to show that for the Neumann Laplacian in certain Rooms and Passages domains the second term of the asymptotic expansion of the counting function is of order $\sqrt{\lambda}$. For the Dirichlet Laplacian our methods only give an upper estimate of the form $\sqrt{\lambda}$. In the second part of the paper, we consider the relationship between Neumann Laplacians on Rooms and Passages domains and Sturm-Liouville operators on the skeleton. 
\end{abstract}
\maketitle

\section{Introduction}

Let $-\Delta_{N,\Omega}$ denote the Neumann Laplacian on a bounded open subset $\Omega$ of $\mathbb{R}^n, n \ge 2,$  and let $E: H^1(\Omega) \rightarrow L^2(\Omega) $ be the canonical  embedding, where $H^1(\Omega)$ is the standard Sobolov space. Then  $(-\Delta_{N,\Omega} +1)^{-1} =
EE^{*},$ and   $ -\Delta_{N,\Omega}$ has a compact resolvent (and thus a discrete spectrum) if and only if $E$ is compact. The compactness, or otherwise, of $E$ is determined by the nature of the boundary $ \partial \Omega$ of $\Omega.$ In \cite{Bur} it was shown that $E$ being compact is equivalent to a \textit{generalised extension property}, which is that \textit{there exists a function space $\Lambda(\mathbb{R}^n)$ on $\mathbb{R}^n$ which is compactly
embedded in $L^2(B)$ for any ball $B \subset \mathbb{R}^n,$ and is
such that there is a continuous extension $
\mathcal{E}:H^1(\Omega) \rightarrow \Lambda(\mathbb{R}^n).$} In
general $ \Lambda(\mathbb{R}^n)$ is a space of arbitrary
smoothness. If $\partial \Omega \in Lip_{\alpha}, 0<\alpha <1,$ i.e., $\partial \Omega$ coincides with the subgraph of a
$Lip_{\alpha}$ function in a neighbourhood of each point,
then $\Lambda(\mathbb{R}^n) = H^{\alpha}(\mathbb{R}^n),$ the
Sobolev space of order $\alpha,$ so that, in this case, there is a reduction of
smoothness in going from $\Omega$ to $\mathbb{R}^n.$ For domains with singular boundaries, like the ``Rooms and Passages" domain defined in Section 2, the compactness of the
embedding $E$ can be taken as a measure of the smoothness of the boundary,
which is otherwise difficult to describe. When $E$ is not compact, \cite{HSS} shows that for any closed subset $S$ of the non-negative real numbers, there exists a modified ``Rooms and Passages" or a ``Comb" domain such that $S$ equals the essential spectrum of $-\Delta_{N,\Omega}$.

Suppose that  $-\Delta_{N,\Omega}$ has a discrete spectrum and denote the number of its eigenvalues less than $\lambda$ by $N_N(\lambda).$ The problem of determining the asymptotic
behaviour of $N_N(\lambda)$ as $\lambda \rightarrow \infty$ has a
long history. If the boundary $\partial \Omega$ of $\Omega$ is
sufficiently smooth, it has long been known that one has the
asymptotic formula
\begin{equation}\label{1}
    N_N(\lambda) = \omega_n(2\pi)^{-n}|\Omega|\lambda^{n/2} +o(\lambda^{n/2}),
\end{equation}
where $\omega_n$ is the volume of the unit ball in $\mathbb{R}^n$ and $|\Omega|$ is the volume of the domain $\Omega$.   For general $\Omega$ the Weyl term
$\omega_n(2\pi)^{-n}|\Omega|\lambda^{n/2}$ dominates, and in the error $R(\lambda):=
N_N(\lambda) - \omega_n(2\pi)^{-n}|\Omega|\lambda^{n/2},$ it is the boundary $\partial
\Omega,$ rather than any other topological or geometrical feature
of $\Omega$ which is dominant. For instance, in \cite{NS}, it
is shown that if $\partial \Omega \in Lip_{\alpha},0<\alpha<1,$ the
remainder $R(\lambda)=O(\lambda^{(n-1)/2\alpha})$ and this is order
sharp. Here and in the following, the  $O(f(\lambda))$- and $o(f(\lambda))$-notation is to be understood in the limit $\lambda\to \infty$.
Moreover, in \cite{EH2}, a domain of von-Koch snowflake
type
 $\Omega$ is exhibited for which $R(\lambda) \asymp
\lambda^{d_o},$ where $d_o$ denotes the outer Minkowski dimension
of the boundary and $A(\lambda)\asymp B(\lambda)$ means that $|A(\lambda)/B(\lambda)|$ is bounded above and below by positive constants as $\lambda\to \infty$. Specific information about the error is hard to
come by. For general domains one of the few available
techniques is the Courant-Weyl variational method involving
``Dirichlet-Neumann bracketing". This effectively reduces the
problem to estimating the corresponding counting functions
$N_D(\lambda)$ and $N_N(\lambda)$ for Dirichlet and Neumann Laplacians on a set of
cubes which cover $\Omega.$ A variant of this technique is used in \cite{NS} in which cubes are replaced by other relatively simple sets for which the Dirichlet and Neumann Laplacian counting functions can be estimated from above and below. The first part of this paper is a contribution to the study of how the
error term $R(\lambda)$ depends on the boundary $\partial \Omega.$
We look in detail at the much studied ``Rooms and Passages"
domain, in which the Weyl formula (\ref{1}) holds, being
particularly concerned with upper and lower bounds for
$R(\lambda)\lambda^{-1/2}.$

In \cite{ES} it was shown that for a rather restricted class of domains $\Omega,$ (which does not include Rooms and Passages) $-\Delta_{N,\Omega}$ has a compact resolvent if and only if a Sturm-Liouville operator defined on the \textit{skeleton} of $\Omega$ has a compact resolvent. Recall that the \textit{skeleton} of an open set $\Omega$ is the complement of the set of points ${\bf x}$ in $\Omega$ for which there exists a unique point ${\bf y}$ on $\partial \Omega$ such that $|{\bf y}-{\bf x}|$ is equal to the distance of ${\bf x}$ from $\mathbb{R}^n \setminus \Omega.$ The result in \cite{ES} was motivated by Theorem 3.3 in \cite{DS} in which $\Omega $ is a horn, whose skeleton is a half-line. In the second part of the paper (Sections 4 and 5) we investigate this problem for general Rooms and Passages domains. 

\section{Rooms \& Passages domains}

We consider a Rooms and Passages (R\&P  for short) domain $\Omega$ defined as the union of square rooms $R_i\ (i~ \textrm{odd})$ of size $h_i\times h_i$ joined by rectangular passages $P_i\ (i~ \textrm{even})$ of size $h_i\times \delta_i$ with $\delta_i<\min\{h_{i-1},h_{i+1}\}$. We assume $h_i\to 0$ as $i\to\infty$. See Figure \ref{fig:example 1}.

\begin{figure}[htbp] 
    \centering
    \includegraphics[width=4in]{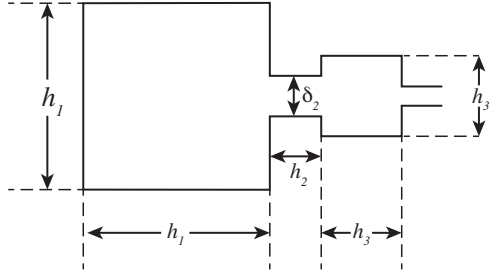}  \vspace{-30pt}
    \caption{The start of a general Rooms and Passages domain.}
    \label{fig:example 1}
\end{figure}

In this section and the next, we further restrict the R\&P domains under consideration by assuming that $h_i=C^i$ and $\delta_i=k C^{i\alpha}$, where $C<1$, $\alpha>1$ and $k$ are constants, with $k<C^{3-2\alpha}$. This guarantees that each passage is narrower than both adjacent rooms.
It follows from the general result in \cite{EH1}, Example 6.1.1, that the embedding $E: H^1(\Omega)\rightarrow L^2(\Omega)$ is compact if and only if $ \alpha < 3.$ 

\begin{proposition} For $\alpha >3$ we have  $0\in \sigma_{ess}(-\Delta_{N,\Omega})$
where $\sigma_{ess}$ denotes the essential spectrum of the operator.
\end{proposition}

\begin{proof}
We give a simple singular sequence proof. By $\Omega_{j}$ we denote the domain consisting of the first $j$ rooms and passages while $T_{j}=\Omega\setminus\Omega_{j}$ denotes the ``tail". Choose a sequence of $C^\infty$ functions $\varphi_j$ which depend only on the $x$-variable such that $\supp{\varphi_j}\subseteq \Omega_{4j}\setminus\Omega_{2j-1}$, 

\be \label{varphiprime}
 \varphi_j=\left\{\begin{array}{cl} 1 & \hbox{on }\ \Omega_{4j-1}\setminus\Omega_{2j}, \\ 
 0 & \hbox{on }\ \Omega_{4j}^c\cup\Omega_{2j-1}, \end{array}\right. \ \hbox{ and }\
 |\varphi_j'| = \left\{\begin{array}{cl} O(C^{-2j}) & \hbox{on }\ \Omega_{2j}\setminus\Omega_{2j-1}, \\ 
 O(C^{-4j})  & \hbox{on }\ \Omega_{4j}\setminus\Omega_{4j-1}.\end{array}\right. 
\ee
In other words, the function $\varphi_j$ is initially zero, increases to $1$ in the $j$-th passage and falls back to zero in the $2j$-th passage. The support of the sequence of the $\varphi_j$ `disappears' into the tail and each $\varphi_j$ satisfies the Neumann boundary condition on $\partial\Omega$. The singular sequence we wish to consider is given by 
\be
f_j(x,y)=\frac{1}{\norm{\varphi_j}_{L^2(\Omega)}} \varphi_j(x)\ \hbox{ for } (x,y)\in\Omega.
\ee
Obviously, $\norm{f_j}_{L^2(\Omega)}=1$, while $f_j$ weakly converges to $0$ in $L^2(\Omega)$ as $j\to\infty$.

We have the following:
\be
\norm{\varphi_j}_{L^2(\Omega)}\asymp \left(|\Omega_{4j}|-|\Omega_{2j}|\right)^{1/2} \ \hbox{ as } j\to\infty
\ee
and
\bea \nonumber
|\Omega_{4j}|&=&\sum_{i=1,\ i\ odd}^{4j-1}C^{2i}+\sum_{i=1,\ i\ even}^{4j}kC^{i(1+\alpha)}\\
&=& C^{2}\frac{1-C^{8j}}{1-C^{4}}+kC^{2(1+\alpha)}\frac{1-C^{4j(1+\alpha)}}{1-C^{2(1+\alpha)}}.
\eea
Therefore,
$$
|\Omega_{4j}|-|\Omega_{2j}|= C^{2}\frac{C^{4j}-C^{8j}}{1-C^{4}}+kC^{2(1+\alpha)}\frac{C^{2j(1+\alpha)}-C^{4j(1+\alpha)}}{1-C^{2(1+\alpha)}}\ \asymp\ C^{4j} \hbox{ as } j\to\infty,
$$
and for large $j$, $\norm{\varphi_j}_{L^2(\Omega)}$ behaves like $C^{2j}$.

As $\varphi_j'$ is supported only on two passages we get from (\ref{varphiprime}),
\be
\norm{\varphi_j'}^2_{L^2(\Omega)} = O( C^{-8j}kC^{4j(1+\alpha)}+C^{-4j}kC^{2j(1+\alpha)})
= O( C^{-2j(1-\alpha)}) \hbox{ as } j\to\infty.
\ee
Hence, as  $j\to\infty$,
\be
\norm{f_j'}_{L^2(\Omega)}=\frac{1}{\norm{\varphi_j}_{L^2(\Omega)}}\norm{\varphi_j'}_{L^2(\Omega)}\ = O(  C^{-j(1-\alpha)}C^{-2j})\ =\ O(C^{-j(3-\alpha)})\ \to 0
\ee
if  $\alpha>3$.
This implies that $(-\Delta)^{1/2}f_j \to 0$ as  $j\to\infty$ although $\norm{f_j}_{L^2(\Omega)}=1$. Therefore, we have $0\in\sigma_{ess}((-\Delta_{N,\Omega})^{1/2})$, which implies $0\in \sigma_{ess}(-\Delta_{N,\Omega})$ when $\alpha>3$.
\end{proof}

\section{Eigenvalue asymptotics of the Dirichlet and Neumann Laplacians}

The special class of R\&P domains introduced in the previous section  will be considered.
Let $\alpha<3$ and denote by $N_D(\lambda), N_N(\lambda),$ respectively, the counting functions of the Dirichlet and Neumann Laplacians on the R\&P  domain $\Omega$.  To determine their asymptotic behaviour,  we shall apply the Dirichlet-Neumann bracketing technique.


We begin this analysis of the spectrum of the Laplacian on an R\&P domain $\Omega$ by discussing the contribution from the tail.
As before, let $\Omega=\Omega_{2M}\cup T_{2M}$ where $\Omega_{2M}$ consists of the first $2M$ rooms and passages and $T_{2M}$ denotes the tail. On applying Theorems 4.6 and 5.1 to Example 6.1.1 in \cite{EH1}, it follows\footnote{By Theorem 4.6, the bound on $K(T_{2M})$ involves a quantity $c(J_1)$, which according to Theorem 5.1 is bounded by $\sqrt{N_a}$. Noting that their constant $C$ corresponds to our $C^{-1}$, Theorem 6.1 and Example 6.1.1 with $k=2M$ then give $N_a\asymp C^{2M(3-\alpha)}$.} that the optimal constant $K\left(T_{2M}\right)$ in the Poincar$\acute{\hbox{e}}$ inequality
\[
\|f - f_{T_{2M}}\|_{L^2(T_{2M})} \le K\left(T_{2M}\right) \|\nabla f \|_{L^2(T_{2M})},  \quad(f \in H^1(T_{2M})),
\]
satisfies
$$
K\left(T_{2M}\right) \le c C^{(3-\alpha)M} 
$$
where $c$ is a positive constant.

It follows that the Neumann Laplacian on $T_{2M}$  will only contribute the trivial eigenvalue $0$  if 
\begin{equation}\label{la} \lambda < (1/c)^2 C^{2(\alpha - 3) M},\quad 
\hbox{ i.e. }\quad M>\frac{ \log\left(c^2 \lambda \right)  }{2(3-\alpha)\log C^{-1}}.
\end{equation} 
Since $\alpha>1$, the tail after $M$ passages has area
$$ \mid T_{2M} \mid = \frac{C^{2+4M}}{1-C^{4}}+\frac{ k C^{2\left(1+\alpha\right)}C^{2M\left(1+\alpha\right)}}{1-C^{2\left(1+\alpha\right)}} \asymp C^{4M} \hbox{ as } M\to\infty.
$$
It follows that  for $M$ as in \eqref{la}
$$\mid T_{2M }\mid = O( \lambda^{-2/(3 - \alpha)})$$
and so \[
|\Omega| = |\Omega_{2M}| + o(\sqrt{\lambda}).
\]

Now, for $M$ satisfying \eqref{la}, the counting function for the Neumann-Laplacian on $\Omega$ differs by at most 1 from that on $\Omega_{2M}$. 
It follows that 
\begin{equation}\label{error}
R(\lambda) - R_{2M}(\lambda)   =   o(\sqrt{\lambda}),
\end{equation} 
where $R(\lambda)$ and  $R_{2M}(\lambda)$ denote the error term for the problem on $\Omega$ and $\Omega_{2M}$, respectively.

As we have that the $(n+1)$-th Neumann eigenvalue is a lower bound for the $n$-th Dirichlet eigenvalue (see \cite{Fil}), for $\lambda$ satisfying \eqref{la}, the tail cannot contribute any Dirichlet eigenvalue, so the same reasoning as for the Neumann case implies that \eqref{error} also holds in the Dirichlet case.

 \subsection{Asymptotics for $ N_N(\lambda)$} 
Our strategy here is to partition the domain and use the Dirichlet Neumann bracketing technique to obtain the required estimates.
In order to obtain these estimates we first  obtain a lower bound for the number of eigenvalues of the Neumann-Laplacian, we partition the rooms into five sections imposing Neumann boundary conditions on the boundary of $\Omega$ and Dirichlet boundary conditions on all artificially introduced internal boundaries (see Figure \ref{fig:6}). An upper bound is obtained by only introducing an artificial boundary to separate the room from the neighbouring passages and imposing Neumann conditions on all the boundaries (see Figure \ref{fig:4}). This is a simple consequence of the variational principle. 

\begin{figure}[htbp] 
    \centering
    \includegraphics[width=2.5in]{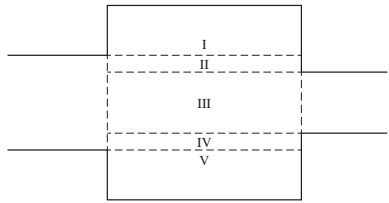} \vspace{-20pt}
    \caption{Artificial boundaries and sub-regions of a room for the lower bound: Neumann conditions on the dotted lines.}
    \label{fig:6}
\end{figure}  
 
\begin{figure}[htbp] 
    \centering
    \includegraphics[width=2in]{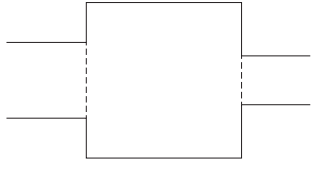} \vspace{-20pt}
    \caption{Artificial boundaries of a room for the upper bound: Dirichlet conditions on the dotted lines.}
    \label{fig:4}
\end{figure}

We first consider the one-dimensional problem on $[-a,a]$ and have the following eigenfunctions and eigenvalues:
\begin{enumerate}
	\item[(a)] Dirichlet conditions at both end points: 
	$$ \psi_m(x)=c\sin{\frac{m\pi(x+a)}{2a}}, \quad \lambda_m=\frac{m^2\pi^2}{4a^2}, \quad m\geq 1.$$
	\item[(b)] Dirichlet conditions at $-a$ and Neumann at $a$: 
	$$ \psi_m(x)=c\sin{\frac{(2m+1)\pi(x+a)}{4a}}, \quad \lambda_m=\frac{(2m+1)^2\pi^2}{16a^2}, \quad m\geq 0.$$
	\item[(c)] Neumann conditions at both end points: 
	$$ \psi_m(x)=c\cos{\frac{m\pi(x+a)}{2a}}, \quad \lambda_m=\frac{m^2\pi^2}{4a^2}, \quad m\geq 0.$$
\end{enumerate}
For the lower estimate for a room, this leads to the following.

\begin{enumerate}
	\item In regions I and V, we have the set of eigenvalues $$\left\{\lambda_{m,n}=\frac{m^2\pi^2}{4a^2}+\frac{(2n+1)^2\pi^2}{16b^2}: \ m,n\geq 0\right\}\quad \hbox{ with } \ a=\frac{C^{j}}{2},\ b=\frac{C^{j}-kC^{\alpha(j-1)}}{4}. $$
	\item In regions II and IV, we have the set of eigenvalues 
	\begin{eqnarray*}
&	\left\{\lambda_{m,n}=\frac{(2m+1)^2\pi^2}{16a^2}+\frac{n^2\pi^2}{4b^2}: \ m\geq 0,\ n\geq 1\right\}\\ 
&\hbox{ with } \ a=\frac{C^{j}}{2},\ b=\frac{k}{4}\left(C^{\alpha(j-1)}-C^{\alpha(j+1)}\right).\end{eqnarray*}
		\item In region III, we have the set of eigenvalues $$\left\{\lambda_{m,n}=\frac{m^2\pi^2}{4a^2}+\frac{n^2\pi^2}{4b^2}: \ m,n\geq 1\right\}\quad \hbox{ with } \ a=\frac{C^{j}}{2},\ b=k\frac{C^{\alpha(j+1)}}{2}. $$
\end{enumerate}

We now need to count the integer lattice points in the first quadrant satisfying $\lambda_{m,n}\leq\lambda$.  By \cite{Hux}, the number of integer lattice points in a plane region $X\cdot R$, where $X$ is a real scaling parameter, is given by
$$\mathcal{N}(X)=AX^2+o(X^{\frac{131}{208}+\eps}) \hbox{ as } X\to\infty$$
for any $\eps>0$, where $A$ is the area of the region $R$. This gives us the following:
\begin{enumerate}
	\item In regions I and V, we have an ellipse with half axes $\frac{2a\sqrt{\lambda}}{\pi}$ and
	$\frac{2b\sqrt{\lambda}}{\pi}$ which is shifted by $-1/2$ in the $y$-direction. As an upper estimate for the area below the $x$-axis we use the area of the rectangle which is subtracted from the area of the quarter ellipse (the error from this can easily be seen to be $o(\sqrt{\lambda}))$. We then add the contributions along the two axes to get
\begin{eqnarray*}
	&& \textrm{card}~\left\{\lambda_{m,n}=\frac{m^2\pi^2}{4a^2}+\frac{(2n+1)^2\pi^2}{16b^2}: \ m,n\geq 0,\ \lambda_{m,n}\leq\lambda\right\}\\
	&&\geq \frac{\pi}{4}\frac{2a\sqrt{\lambda}}{\pi}\frac{2b\sqrt{\lambda}}{\pi}-\frac{1}{2}\frac{2a\sqrt{\lambda}}{\pi}-o(\sqrt{\lambda})+\frac{2a\sqrt{\lambda}}{\pi}+\frac{2b\sqrt{\lambda}}{\pi}\\
	&&=\frac{ab}{\pi}\lambda+\frac{(2b+a)\sqrt{\lambda}}{\pi}-o(\sqrt{\lambda}).
\end{eqnarray*}
Here, $\textrm{card}~A$ denotes the cardinality of the finite set $A$.
	\item In regions II and IV, we  again  have an ellipse with half axes $\frac{2a\sqrt{\lambda}}{\pi}$ and
	$\frac{2b\sqrt{\lambda}}{\pi}$ which  this time  is shifted by $-1/2$ in the $x$-direction. As an upper estimate for the area left of the $y$-axis we use the area of the rectangle which is subtracted from the total area. We then add the contributions along the $y$-axis to get
\begin{eqnarray*}
	&& \textrm{card}~\left\{\lambda_{m,n}=\frac{(2m+1)^2\pi^2}{16a^2}+\frac{n^2\pi^2}{4b^2}: \ m\geq 0,\ n\geq 1,\ \lambda_{m,n}\leq\lambda\right\}\\
	&&\geq \frac{\pi}{4}\frac{2a\sqrt{\lambda}}{\pi}\frac{2b\sqrt{\lambda}}{\pi}-\frac{1}{2}\frac{2b\sqrt{\lambda}}{\pi}-o(\sqrt{\lambda})+\frac{2b\sqrt{\lambda}}{\pi}\\
	&&=\frac{ab}{\pi}\lambda+\frac{b\sqrt{\lambda}}{\pi}-o(\sqrt{\lambda}).
\end{eqnarray*}

\item In region III, we   count the lattice points in an ellipse with half axes $\frac{2a\sqrt{\lambda}}{\pi}$ and
	$\frac{2b\sqrt{\lambda}}{\pi}$ to give
\begin{eqnarray*}
	&& \textrm{card}~\left\{\lambda_{m,n}=\frac{m^2\pi^2}{4a^2}+\frac{n^2\pi^2}{4b^2}: \ m,n\geq 1,\ \lambda_{m,n}\leq\lambda\right\}\\
	&&\geq \frac{\pi}{4}\frac{2a\sqrt{\lambda}}{\pi}\frac{2b\sqrt{\lambda}}{\pi}-o(\sqrt{\lambda})\ =\ \frac{ab}{\pi}\lambda-o(\sqrt{\lambda}).
\end{eqnarray*}
\end{enumerate}
Collecting these results,   we obtain  a lower estimate for  the contribution of the $j$-th room to the counting function of the form 
$$N_j\geq\frac{C^{2j}}{4\pi}\lambda +\left[2C^{j}-\frac{k}{2}\left(C^{\alpha(j+1)}+ C^{\alpha(j-1)}\right)\right]\frac{\sqrt{\lambda}}{\pi}-o(\sqrt{\lambda}).$$

For the upper estimate of the counting function, we simply need to consider the eigenvalues 
$$\left\{\lambda_{m,n}=\frac{m^2\pi^2}{4a^2}+\frac{n^2\pi^2}{4b^2}: \ m,n\geq 0\right\}$$
of the Neumann-Laplacian on the square, where $a=b=\frac{C^{j}}{2}$. In order to count the integer lattice points, we
take the area of the ellipse with half axes $\frac{2a\sqrt{\lambda}}{\pi}$ and $\frac{2b\sqrt{\lambda}}{\pi}$  and add the  additional points along the $x$- and $y$-axes:
\begin{eqnarray*}
	&& \textrm{card}~\left\{\lambda_{m,n}=\frac{m^2\pi^2}{4a^2}+\frac{n^2\pi^2}{4b^2}: \ m,n\geq 0,\ \lambda_{m,n}\leq\lambda\right\}\\
	&&\leq \frac{\pi}{4}\frac{2a\sqrt{\lambda}}{\pi}\frac{2b\sqrt{\lambda}}{\pi}+\frac{2a\sqrt{\lambda}}{\pi}+\frac{2b\sqrt{\lambda}}{\pi}+o(\sqrt{\lambda})\\	&&=\frac{ab}{\pi}\lambda+\frac{2a\sqrt{\lambda}}{\pi}+\frac{2b\sqrt{\lambda}}{\pi}+o(\sqrt{\lambda}).
\end{eqnarray*}
As an upper estimate for the contribution of the $j$-th room, we therefore get
$$N_j\leq\frac{C^{2j}}{4\pi}\lambda +\frac{2C^{j}}{\pi}\sqrt{\lambda}+o(\sqrt{\lambda}).$$

The calculations involving the first room are a little different and we only use three partitions (see Figure \ref{fig:5}).
\begin{figure}[htbp] 
    \centering
    \includegraphics[width=2.5in]{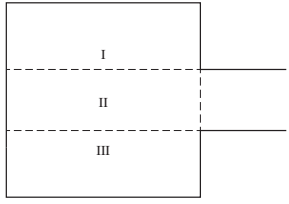} \vspace{-20pt}
    \caption{Subdivisions of the first room for the lower bound: Neumann conditions on the dotted lines.}
    \label{fig:5}
\end{figure} 
\begin{enumerate}
	\item In regions I and III, we have the set of eigenvalues $$\left\{\frac{m^2\pi^2}{4a^2}+\frac{(2n+1)^2\pi^2}{16b^2}: \ m,n\geq 0\right\}\quad \hbox{ with } \ a=\frac{C}{2},\ b=\frac{C-kC^{2\alpha}}{4}. $$
	\item In region II, we have the set of eigenvalues $$\left\{\frac{(2m+1)^2\pi^2}{16a^2}+\frac{n^2\pi^2}{4b^2}: \ m\geq 0,\ n\geq 1\right\}\quad \hbox{ with } \ a=\frac{C}{2},\ b=\frac{kC^{2\alpha}}{2}. $$
\end{enumerate}
Proceeding as for the other rooms we  again need to count integer lattice points in the first quadrant.
\begin{itemize}
	\item In regions I and III, we have 
\begin{eqnarray*}
	&& \textrm{card}~\left\{\lambda_{m,n}=\frac{m^2\pi^2}{4a^2}+\frac{(2n+1)^2\pi^2}{16b^2}: \ m,n\geq 0,\ \lambda_{m,n}\leq\lambda\right\}\\
	&&\geq \frac{ab}{\pi}\lambda+\frac{(2b+a)\sqrt{\lambda}}{\pi}-o(\sqrt{\lambda}).
\end{eqnarray*}

	\item In region II, we have 
\begin{eqnarray*}
	&&\textrm{card}~\left\{\lambda_{m,n}=\frac{(2m+1)^2\pi^2}{16a^2}+\frac{n^2\pi^2}{4b^2}: \ m\geq 0,\ n\geq 1,\ \lambda_{m,n}\leq\lambda\right\}\\
	&&\geq \frac{ab}{\pi}\lambda+\frac{b\sqrt{\lambda}}{\pi}-o(\sqrt{\lambda}).
\end{eqnarray*}
\end{itemize}
Combining these we get  a lower estimate for the contribution of the first room
\begin{eqnarray*}
	N_1&\geq&
	\frac{\lambda C^{2}}{4\pi } +\left(2C-\frac{k}{2}C^{2\alpha}\right)\frac{\sqrt{\lambda}}{\pi}
-o(\sqrt{\lambda}).
\end{eqnarray*}

 We next sum over the rooms (omitting the $o(\sqrt{\lambda})$-term for convenience). Denoting  the volume of the first $M$ rooms by $V(R_{M})$, we see
\begin{eqnarray*}
	\sum_{j=1,  j\; odd}^{2M}  N_j&\geq& \sum_{j=1,  j\; odd}^{2M} \frac{\lambda C^{2 j}}{4 \pi }\\
&& \hspace{-35pt} + \left(2\sum_{j=1,  j\; odd}^{2M}C^{j}-\frac{k}{2} \left(C^{2\alpha}+\sum_{j=3,  j\; odd}^{2M}C^{\alpha(j-1)}\left(1+C^{2\alpha}\right) \right)\right) \frac{\sqrt \lambda}{\pi}\\
 &\hspace{-50pt}=& \hspace{-25pt} \frac{V(R_{M})}{4 \pi}  \lambda + \left(2C\frac{ 1-C^{2M}}{1-C^{2}}-\frac{k}{2}\cdot
\frac{2C^{2\alpha}-C^{2\alpha M}-C^{2\alpha(M+1)}}{1-C^{2\alpha}}\right)\frac{ \sqrt \lambda }{\pi}.
\end{eqnarray*}
Also,
$$\sum_{j=1,  j\; odd}^{2M} N_j\leq
\sum_{j=1, j\; odd}^{2M} \left(\frac{\lambda  C^{2 j}}{4 \pi}+
 \frac{2 \sqrt \lambda C^j}{\pi }\right)=\frac{V(R_{M})}{4 \pi}  \lambda + \frac{ 2 \sqrt \lambda}{\pi C} \left(\frac{C^{2}-C^{2\left(M+1\right)}}{1-C^{2}}\right).
$$
 
 We now proceed to estimate the counting function in a passage.
  Here, 
  $$a=\frac{C^{j}}{2},\;\;b=\frac{k C^{\alpha j}}{2}\ \hbox{ and }\ \lambda_{m,n}=\frac{\pi^2 m^2}{C^{2j}} + \frac{\pi^2n^2}{ C^{2 \alpha j} k^2} ,\ m\geq 1, n\geq 0.$$
 Then, since the lattice point counting estimate  is on the ellipse with semi axes
 $\frac{2a\sqrt{\lambda}}{\pi}$ and $\frac{2b\sqrt{\lambda}}{\pi}$,  we have
 $$N_j \geq \frac{\lambda}{\pi} a b  -o\left(\sqrt \lambda\right) + \frac{2a\sqrt{\lambda}}{\pi}=\frac{k C^{\left(1+\alpha\right)j }}{4\pi }\lambda+ \frac{ C ^{ j}}{\pi } \sqrt \lambda - o\left(\sqrt \lambda\right).$$
  For the upper bound we count the eigenvalues
  $$\lambda_{m,n}=\frac{\pi^2 m^2}{C^{2j}} + \frac{\pi^2n^2}{ C^{2 \alpha j} k^2} ,\ m\geq 0, n\geq 0.$$
  Thus
 $$N_j\leq \frac{\lambda}{\pi} a b  +o\left(\sqrt \lambda\right) + \frac{2a\sqrt{\lambda}}{\pi} + \frac{2b\sqrt{\lambda}}{\pi} =
 \frac{k C^{\left(1+\alpha\right)j}}{4\pi}\lambda + \left( \frac{kC^{\alpha j}}{\pi }+\frac{ C^j}{\pi}\right)\sqrt \lambda + o\left(\sqrt \lambda\right). $$

Summing, we obtain  bounds  for the counting function in the passages
\begin{eqnarray*}
\sum_{j=1, \; even}^{2M} N_j&\geq& \sum_{j=1, \; even}^{2M}\frac{kC^{\left(1+\alpha \right)j}}{4 \pi }\lambda + \frac{C^{j}}{\pi }\sqrt \lambda \ =\ \frac{V(P_M)}{4 \pi}  \lambda + \sum_{i=1}^{M}\frac{C^{2  i}}{\pi }\sqrt \lambda\\
&=&\frac{V(P_M)}{4 \pi}  \lambda +   \frac{\sqrt \lambda}{\pi} \frac{C^{2}-C^{2 \left(M+1\right)}}{1-C^{2}},
\end{eqnarray*}
where $V(P_M)$ denotes the volume of the first $M$ passages.
We have the upper bound 
\begin{eqnarray*}
\sum_{j=1, \; even}^{2M} N_j&\leq& 
\frac{V(P_M)}{4 \pi}  \lambda + \sum_{j=1,\; j\;even}^{2M}\left(
\frac{kC^{\alpha j}}{\pi } +\frac{C^j}{\pi }\right) \sqrt \lambda\\
&&=\frac{V(P_M)}{4 \pi}  \lambda + \frac{\sqrt \lambda}{\pi }  \left(k\sum_{i=1}^M C^{2\alpha i}+\sum_{i=1}^M C^{2 i}\right)\\
&&=
 \frac{V(P_M)}{4 \pi} \lambda+
 \left(k\frac{ C^{2\alpha}-C^{2\alpha \left(1+M\right)}}{1-C^{2\alpha}} + 
\frac{C^{2}-C^{2 \left(1+M\right)}}{1-C^{2}}\right)\frac{\sqrt{\lambda}}{\pi}.
\end{eqnarray*}

We now are in a position to estimate the bounds of the counting function for the domain $ \Omega_{2M}$
Summing  the contributions from the rooms and the passages, we get the   lower estimate
\begin{equation*}
\frac{\mid \Omega_{2M}\mid}{4 \pi} \lambda + \left(\frac{ 2C +C^{2}-2C^{2M+1}-C^{2M+2}}{1-C^{2}}
	-\frac{k}{2}\frac{2C^{2\alpha}-C^{2\alpha M}-C^{2\alpha(M+1)}}{1-C^{2\alpha}}
	\right)
	\frac{ \sqrt \lambda }{\pi}
\end{equation*}
 with the upper estimate   given by
 $$
\frac{\mid \Omega_{2M}\mid}{4 \pi} \lambda + \left( \frac{ 2C+C^{2}-2C^{2M+1}-C^{2M+2}}{1-C^{2}}+k \frac{  C^{2 \alpha}-C^{2\alpha \left(1+M\right)}}{1-C^{2\alpha}}\right)\frac{\sqrt \lambda}{\pi}.
 $$
 As $M\to \infty$, since there is no contribution from the tail  and $C<1$ we get
 $$
 \frac{ \sqrt \lambda}{\pi}
 \left(\frac{ 2C+C^{2}}{1-C^{2}}
	-k\frac{C^{2\alpha}}{1-C^{2\alpha}}
	\right)
 \leq N_N(\lambda)
  -\frac{\mid \Omega\mid}{4 \pi} \lambda \leq \frac{\sqrt \lambda}{\pi} \left( \frac{2C+C^{2}}{1-C^{2}}+\frac{k C^{2\alpha}}{1-C^{2\alpha}}\right).
 $$
 These yield the following result.
 \begin{theorem}
 As $\lambda \to \infty$,
 $$ C_1\frac{\sqrt \lambda}{\pi} 
\leq N_N(\lambda)
  -\frac{\mid \Omega\mid}{4 \pi} \lambda \leq    C_2
 \frac{\sqrt \lambda}{\pi}
 $$
 where $$C_1=\frac{ 2C+C^{2}}{1-C^{2}}
	-k\frac{C^{2\alpha}}{1-C^{2\alpha}}>0, \quad C_2=\frac{2C+C^{2}}{1-C^{2}}+\frac{k C^{2\alpha}}{1-C^{2\alpha}} $$ and $$C_2-C_1=2k \frac{C^{2\alpha}}{1-C^{2\alpha}}\to 0 \ \hbox{ as }\ k\to 0.
 $$
Therefore, the error is precisely of order $\sqrt \lambda$ and in the limit $k\to 0$, we obtain the precise constant.
\end{theorem}

\subsection{Asymptotics for $N_D(\lambda)$}

We now estimate bounds for the counting function of   the Dirichlet-Laplacian. In order to do this we again look separately  at  the rooms and  the passages. The lower bound is   obtained  by choosing Dirichlet conditions on all the boundaries. The set of eigenvalues is given by
$$\left\{\frac{m^2\pi^2}{4a^2}+\frac{n^2\pi^2}{4b^2}: \ m,n\geq 1\right\}\quad \hbox{ with } \ a=\frac{C^{j}}{2},\ b=\frac{1}{2}\left\{\begin{array}{ll} C^{j} & \hbox{for rooms,} \\ 
k C^{\alpha j}& \hbox{for passages.}\end{array}\right.  $$
Therefore, the contribution to the lower bound from each room or passage is $\frac{ab}{\pi}\lambda-o(\sqrt{\lambda})$. Adding all contributions, we get the lower bound
$$N_D(\lambda)\geq\frac{\mid \Omega\mid}{4 \pi} \lambda +o(\sqrt{\lambda}).$$
To get an upper bound in the first room, we choose Dirichlet conditions on three sides of the boundary and Neumann conditions on the right side. This gives us the set of eigenvalues  
$$\left\{\frac{(2m+1)^2\pi^2}{16a^2}+\frac{n^2\pi^2}{4b^2}: \ m\geq 0,\ n\geq 1\right\}\quad \hbox{ with } \ a=b=\frac{C}{2}. $$
Counting lattice points gives
$\frac{ab}{\pi}\lambda+\frac{b\sqrt{\lambda}}{\pi}+o(\sqrt{\lambda}),$ so
as an upper estimate for the contribution of the first room, we get
$$\frac{\lambda C^{2}}{4\pi } + \frac{C}{2}\frac{\sqrt{\lambda}}{\pi}
+o(\sqrt{\lambda}).$$
For the remaining rooms and passages we choose Dirichlet conditions on the horizontal boundaries and Neumann conditions on the vertical ones. This gives us the set of eigenvalues
$$\left\{\frac{m^2\pi^2}{4a^2}+\frac{n^2\pi^2}{4b^2}: \ m\geq 0,\ n\geq 1\right\}. $$
Again counting lattice points gives
$\frac{ab}{\pi}\lambda+\frac{2b\sqrt{\lambda}}{\pi}+o(\sqrt{\lambda}).$ 
Summing all contributions yields  as the upper estimate (omitting the $o(\sqrt{\lambda})$-term).
\bea \nonumber
&& \frac{|\Omega|}{4\pi}\lambda +  \left(\frac{C}{2}+\sum_{j=3, j\  odd}^{2M}C^{j}+k\sum_{j=1, j\  even}^{2M}C^{j\alpha} \right) \frac{\sqrt{\lambda}}{\pi} \\ \nonumber
&& = \frac{|\Omega|}{4\pi}\lambda +  \left(C\frac{1-C^{2M}}{1-C^{2}}-\frac{C}{2}+kC^{2\alpha}\frac{1-C^{2\alpha M}}{1-C^{2\alpha}} \right) \frac{\sqrt{\lambda}}{\pi}. 
\eea
Therefore, letting $M\to\infty$, we find the following result.
\begin{theorem}
As $\lambda \to \infty$
$$ o(\sqrt{\lambda})
\leq N_D(\lambda)
  -\frac{\mid \Omega\mid}{4 \pi} \lambda \leq    
 \left(\frac{C(C^2+1)}{2(1-C^2)}+\frac{k}{C^{-2\alpha}-1}\right) \frac{\sqrt \lambda}{\pi}+o(\sqrt{\lambda}).
 $$
 \end{theorem}
We remark that   
	 the lower bound is given by having Dirichlet boundary conditions everywhere which does not give a $\sqrt{\lambda}$-term. Therefore, we can only get an $o(\sqrt{\lambda})$ error estimate and not determine the sign of the $\sqrt{\lambda}$-term.

\section{A related problem on the skeleton}

We define $ {\bf y} \in \partial \Omega $ to be a near point of ${\bf x} \in \Omega$ if $ |{\bf x} - {\bf y}| = \textrm{dist}({\bf x}, \partial \Omega),$ the distance of ${\bf x}$ to the boundary of $\Omega.$ Therefore, denoting by ${\mathcal N}(\bf x)$ the set of near points of ${\bf x}$, the skeleton $\Gamma$ of $\Omega$ is the set
\[
{\mathcal{S}}(\Omega) := \{ {\bf x} \in \Omega: \textrm{card}~{\mathcal N}({\bf x}) >1\}.
\]
It follows that the skeleton of the R\&P  domain $\Omega$ is the union of a sequence of line segments and parabolic arcs, $\Gamma= \{e_j\}_{j \in \mathbb{N}}$ say, connecting points in $\Omega;$ see Figure \ref{fig3}. For any ${\bf x} \in \Omega$, there exists ${\bf t} \in \Gamma$ such that ${\bf x} $ lies on one of two line segments $C_+({\bf t}), C_-({\bf t})$ connecting ${\bf t}$ to its 2 near points ${\bf y}_+, {\bf y}_-$: set $\tau: \Omega \rightarrow \Gamma , {\bf x} \rightarrow {\bf t}$. If $\tau({\bf x}) = {\bf t} \in e \in \Gamma$, we may therefore define the following co-ordinate system on $\tau^{-1}(e)$:
\begin{equation}\label{S1}
{\bf x} = {\bf x}(\sigma,s),\ \ \tau({\bf x}) = {\bf t} = {\bf t}(\sigma),\ \ s \in (-l(\sigma), l(\sigma)),
\end{equation}
where $\sigma$ denotes arc length along $e$, $s$ is the distance from ${\bf t}$ to ${\bf x}$ along $C_+({\bf t}(\sigma))$ or $C_-({\bf t}(\sigma)),$ $l(\sigma)$ is the length of $C_+({\bf t}(\sigma))$ and $ C_-({\bf t}(\sigma))$ and we set $0 \le s \le l({\bf t}(\sigma))$ along $C_+({\bf t}(\sigma))$ and $ -l({\bf t}(\sigma))\le s \le 0$ along $C_-({\bf t}(\sigma))$. With ${\bf x}= (x,y)$ and $\tau({\bf x}) = (\tau_1({\bf x}), \tau_2({\bf x}))$ it is shown in \cite[(2.4)]{ES}, that the determinant of the Jacobian
\begin{equation}\label{S2}
J:=\left|\frac{\partial(x,y)}{\partial (\sigma, s)}\right| = \frac{1}{|\nabla \tau (\sigma, s)|} := \frac{1}{[|\nabla \tau_1({\bf x})|^2 + |\nabla \tau_2({\bf x})|^2]^{1/2}}.
\end{equation} 
For a measurable subset $\Gamma_0$ of $\Gamma$ and $\Omega_0 := \tau^{-1}(\Gamma_0)$, it then follows that, for any $f \in L^2(\Omega_0)$ with $f=0 $ outside $\Omega_0$,
\begin{equation}\label{S3}
\int_{\Omega_0} f({\bf x}) d {\bf x} = \sum_{j \in \mathbb{N}} \int_{e_j}d \sigma \int_{-l(\sigma)}^{ l(\sigma)} f(\sigma, s) \frac{1}{|\nabla\tau(\sigma,s)|} ds.
\end{equation}
This implies, in particular, that for $f=F \circ \tau$ with $F \in L^2(\Gamma_0),$
\begin{eqnarray}\nonumber
\int_{\Omega_0} F \circ \tau({\bf x}) d {\bf x}& = & \sum_{j \in \mathbb{N}} \int_{e_j}F(\sigma) d \sigma \int_{-l(\sigma)}^{ l(\sigma)}  \frac{1}{|\nabla\tau(\sigma,s)|} ds = \int_{\Gamma_0} F(\sigma) \alpha(\sigma) d \sigma,
\end{eqnarray}
where
\begin{equation}\label{S5}
\alpha(\sigma) := \int_{-l(\sigma)}^{ l(\sigma)}  \frac{1}{|\nabla\tau(\sigma,s)|} ds .
\end{equation}
The integral
\begin{equation}\label{S6}
\beta(\sigma) := \int_{-l(\sigma)}^{ l(\sigma)}  |\nabla\tau(\sigma,s)| ds 
\end{equation}
will also feature in certain specific regions of $\Omega.$

 \begin{figure}[H] 
      \begin{minipage}[t]{0.45\linewidth}
    \centering
    \includegraphics[width=2.0in]{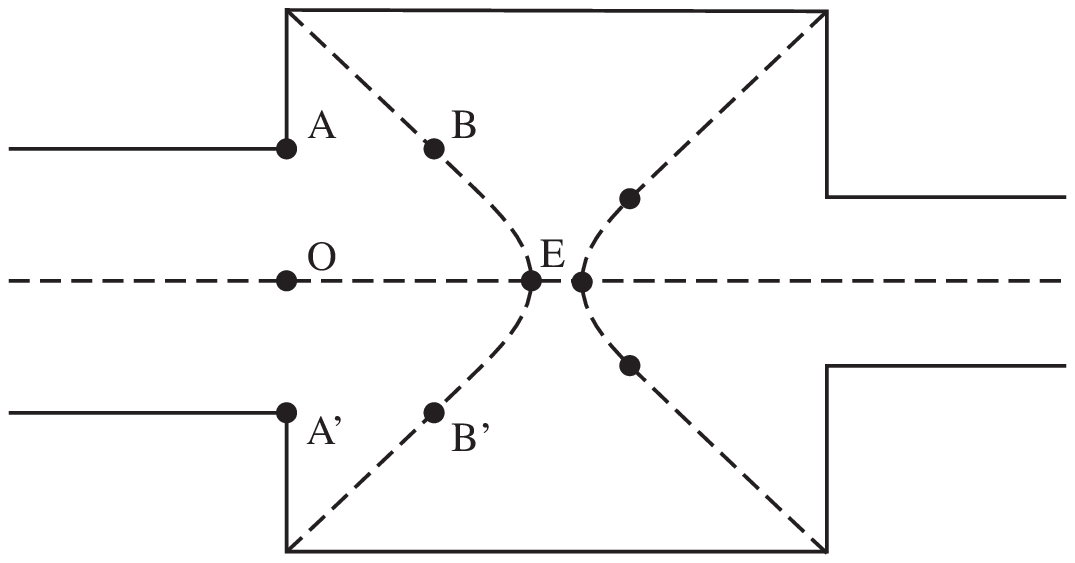}  \vspace{-1pt}
    \caption{The skeleton.}   \label{fig3}
    \end{minipage}
     \begin{minipage}[t]{0.45\linewidth}
    \centering
    \includegraphics[width=2.0in]{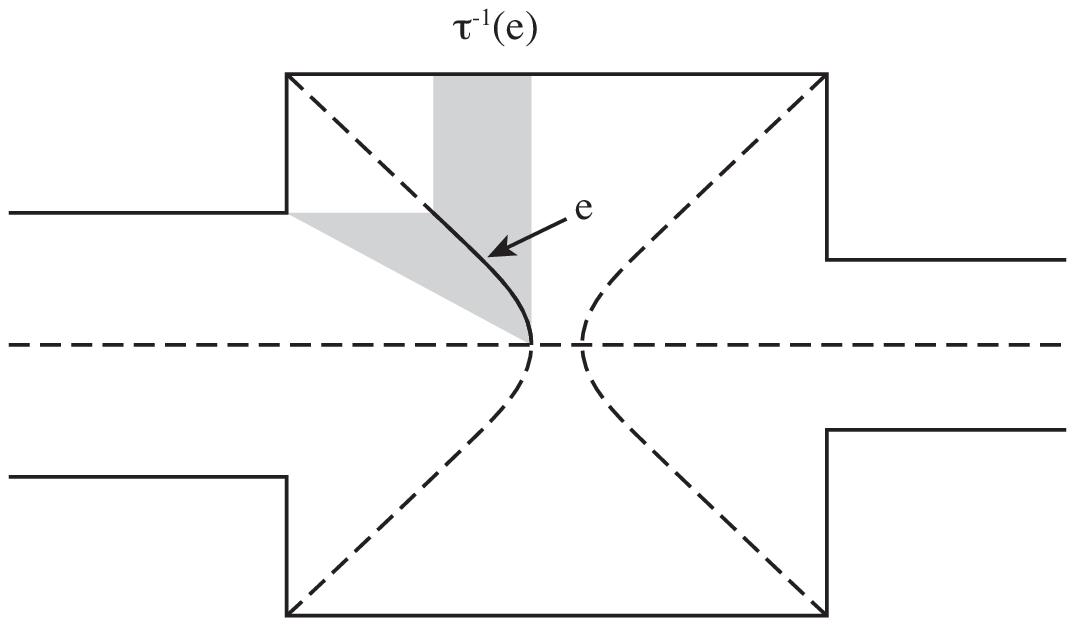}  \vspace{-1pt}
    \caption{ $\tau^{-1}(e) $ for a   parabolic edge.
    }   \label{fig2}
    \end{minipage}%
  
 \end{figure}

We shall be considering a general R\&P domain hereafter, and not the special case of Sections 2 and 3. Therefore we allow for the possibility that the Neumann Laplacian does not have a discrete spectrum.
Our first task is to make explicit the change of co-ordinates \eqref{S1} in each region $ \tau^{-1}(e), e \in \Gamma$ and then determine the map $\tau$. The edges fall into 3 groups which have to be handled separately. In what follows below, we denote the height of a room by $h$ and of a passage by $\delta$. 

\noindent{\bf Group 1} This consists of edges which are either in a passage or lie in the centre of a room with adjacent parabolic edges.  Here, $\sigma=x$, $s=\pm y$, so the determinant of the Jacobian in (\ref{S2}) equals $1$ and $$l(\sigma)=\left\{\begin{array}{ll} h/2 & \hbox{ in rooms},\\   \delta/2 & \hbox{ in passages}. \end{array}  \right.$$
Hence, 
\be\label{abpass}\alpha(\sigma)= \beta(\sigma) = \left\{\begin{array}{ll} h & \hbox{ in rooms},\\   \delta & \hbox{ in passages}. \end{array}  \right.\ee

\noindent {\bf Group 2} These are the edges in a room which are straight line segments along the diagonals. 
In Figure \ref{fig3} with the origin at O, the edge on the diagonal of the square given by $0<x<\frac{h-\delta}{2}$ and $y>\delta/2$ lies in this group and the analysis that follows is typical for all edges in this group.  First, consider the triangle below the bisecting line, i.e. $\delta/2<y<h/2-x$. Here, we re-parameterize points $(x,y)\in\Omega$ by $(\sigma,s)$, where $\sigma$ is the arc length along the skeleton measured from the corner and $s$ is the negative horizontal distance of the point from the skeleton. Thus,
$$(x,y) = \left(\frac{\sigma}{\sqrt{2}}+s, - \frac{\sigma}{\sqrt{2}}+\frac{h}{2}\right),\quad (\sigma,s)= \left(\sqrt{2}(\frac{h}{2}-y), -\frac{h}{2}+x+y\right).$$

In the triangle above the diagonal, where $\frac{h}{2}-x<y<\frac{h}{2}$, we choose $s$ to be the vertical distance to the skeleton. Here,
$$(x,y) = \left(\frac{\sigma}{\sqrt{2}}, - \frac{\sigma}{\sqrt{2}}+\frac{h}{2}+s\right),\quad (\sigma,s)= \left(\sqrt{2}x, -\frac{h}{2}+x+y\right).$$

We note that in this whole square, we have that 
$$\tau(x,y)=\left(\frac{\sigma}{\sqrt{2}},\frac{h}{2}-\frac{\sigma}{\sqrt{2}}\right),$$
with $0<\sigma<\frac{h-\delta}{\sqrt{2}}$. Moreover, the determinant of the Jacobian $J=1/\sqrt{2}$ and $l(\sigma)=\sigma/\sqrt{2}$. Therefore,
\be \label{abcorner}
\alpha(\sigma)=\sigma,\quad \beta(\sigma) =2 \sigma.
\ee

\noindent {\bf Group 3} These are edges which are such that every point on the edge has a re-entrant corner as one of its two near points. Thus in Figure \ref{fig3}, the parabolic edges BE and the line segment OE belong to this group associated with the re-entrant corner A. We consider the parabolic edge BE. For $y>0$ this is determined by $|AQ|=|QQ'|$, where $Q=(x_0,y_0)$ is a point on the parabola and $Q'=(x_0,h/2)$. This gives 
$$x_0^2+\left(y_0-\frac{\delta}{2}\right)^2=\left(y_0-\frac{h}{2}\right)^2.$$
After a little algebra this leads to 
\be \label{parab}
y_0=-\frac{x_0^2}{h-\delta}+\frac14(h+\delta).
\ee
In particular, the parabola intersects the x-axis at $E=\left(\frac{\sqrt{h^2-\delta^2}}{2},0\right)$.

Consider the part of the domain emanating from the re-entrant corner to the parabolic part of the skeleton. Let $\tau(x,y)=(x_0,y_0)$. Then in addition to lying on the parabola, $(x_0,y_0)$ satisfies
\be\label{line}
y_0=\frac{y-\delta/2}{x}x_0+\frac{\delta}{2}.
\ee

The arc length  along the parabola
\be
\sigma=\int_{\frac{1}{2}(h-\delta)}^{x_0} \sqrt{1+\left(\frac{-2x}{h-\delta}\right)^2}\ dx=\frac{h-\delta}{2}\int_{1}^{t_0} \sqrt{1+t^2}\ dt,  
\ee  
with  $t_0=\frac{2x_0}{h-\delta}$.
A straightforward  calculation gives that
 \be
\sigma=
\frac{1}{4} (h-\delta) \left(t_0\sqrt{t_0^2+1} +\sinh
   ^{-1}(t_0)-\sqrt{2}-\sinh ^{-1}(1)\right).
\ee
In particular, the length of the parabolic edge 
is
\begin{equation*}
|CE| 
=  \frac{\sqrt{2}}{4} \sqrt{h(h+\delta)}-\frac{\sqrt{2}}{4}(h-\delta) + \frac{h-\delta}{4}
\left( \sinh^{-1}\left(\sqrt{\frac{h+\delta}{h-\delta}}\right) -\sinh^{-1}\left(1\right)\right).
\end{equation*}
We now use \eqref{parab} and \eqref{line} to express $t_0$ in terms of $x$ and $y$. Eliminating $y_0$ in \eqref{parab}, we get
$$ (h-\delta)\left(\frac{y-\delta/2}{x}x_0+\frac{\delta}{2}-\frac14(h+\delta)\right)+x_0^2=0.$$
This yields
\be
t_0^2+\frac{2y-\delta}{x}t_0-1=0,
\ee
so that 
\be
t_0=\frac{\frac{\delta}{2}-y}{x}+\sqrt{\left(\frac{\frac{\delta}{2}-y}{x}\right)^2+1}.
\ee
For the distance $s$ from $(x,y)$ to $(x_0,y_0)$ this then gives
$$
s^2=  \left(x-\frac{h-\delta}{2}t_0\right)^2 +\left( y- \frac{y-\delta/2}{x}\frac{h-\delta}{2}t_0-\frac{\delta}{2}\right)^2
=  \frac{\left(t_0^2+1\right)^2 \left(x-\frac{h-\delta}{2}t_0\right)^2}{4 t_0^2}
$$
and hence
$$
s=-\frac{\left(t_0^2+1\right) \left(x-\frac{h-\delta}{2}t_0\right)}{2 t_0}.
$$

We need to calculate the determinant of the Jacobian $\frac{\partial(\sigma,s)}{\partial(x,y)}$. Several terms in this Jacobian vanish and we have
$$\left|\frac{\partial(\sigma,s)}{\partial(x,y)}\right| = \left|\frac{\partial \sigma}{\partial t_0} \frac{\partial t_0}{\partial y} \frac{\partial s}{\partial{x}}\right|.$$
A calculation gives
$$\frac{\partial \sigma}{\partial t_0}=
\frac{1}{2} \sqrt{t_0^2+1} (h-\delta),
\quad \frac{\partial s}{\partial{x}}=
-\frac{t_0^2+1}{2 t_0},
\quad \frac{\partial t_0}{\partial y}=-\frac1x
-\frac{\delta-2 y}{x\sqrt{(\delta-2 y)^2+4 x^2}}.
$$
Therefore,
$$
\left|\frac{\partial(\sigma,s)}{\partial(x,y)} \right|
= \frac{\left(t_0^2+1\right)^{3/2} (h-\delta) \left(\frac{ (\delta-2 y)}{\sqrt{(\delta-2 y)^2+4
   x^2}}+1\right)}{4 t_0 x}.$$
Substitution of $t_0$ in terms of $x$ and $y$ gives
\bea \left|\frac{\partial(\sigma,s)}{\partial(x,y)}\right| &
=& 
    \sqrt{2} (h-\delta) \frac{\left((\delta-2 y)^2+4 x^2\right)^{\frac14}}{\left(\sqrt{(\delta-2 y)^2+4 x^2}-\delta+2
   y\right)^{3/2}} .
\eea

To analyse the behaviour of $J^{-1}=\left|\frac{\partial(\sigma,s)}{\partial(x,y)}\right|$ near the re-entrant corner, let $$x= r\cos\theta, \quad y=\frac{\delta}{2}-r\sin\theta \quad (i.e.\ \delta-2y = 2r\sin\theta).$$
Then
$$ \left|\frac{\partial(\sigma,s)}{\partial(x,y)}\right|= \sqrt{2} (h-\delta) \frac{\left(4 r^2\right)^{\frac14}}{\left(2r-2r\sin\theta\right)^{3/2}} =  \frac{h-\delta}{\sqrt{2}r\left(1-\sin\theta\right)^{3/2}},$$ 
so $J^{-1}$ behaves like $1/r$ where $r$ is the distance from the corner. This behaviour of $J^{-1}$ implies that on the parabolic edge 
\begin{equation}\label{30A}
\alpha(\sigma)< \infty, \hbox{ but } \beta(\sigma)=\infty.
\end{equation}
For this reason, we will need to make sure that the weight $\beta$ does not appear in the analysis on those edges (like the parabolic edges) whose points have re-entrant near points. How we do this will be made apparent in the next section.

\section{A Sturm-Liouville operator}
We denote the set of edges of $\Gamma$ which belong to groups 1 and 2 by $\Gamma_{\reg}$ and those in Group 3 by $\Gamma_{\sing}$. Note that $\Gamma_{\sing}$ consists of the parabolic edges and ones like the edge OE in Figure \ref{fig3} which connect an end of a passage and the parabolic edges. The map $\tau$ in (\ref{S2}) maps a re-entrant corner onto every point on a singular edge, which motivates us to define any function $f$ on $ e \in \Gamma_{\sing}$ to be constant.

The underlying Hilbert spaces on $\Gamma$  are as follows:
\be
\Lt^2(\Gamma) = \bigoplus_{e \in \Gamma_{\reg}} L^2(e;\alpha(\sigma) d \sigma)\bigoplus_{e \in \Gamma_{\sing}} \C {\bf 1}_e,
\ee
where ${\bf 1}_e$ is the characteristic function of the edge $e$,  $ L^2(e;\alpha(\sigma) d \sigma)$ is the weighted 
Lebesgue space with inner-product
\be
\int_e f(\sigma) \overline{g(\sigma)} \alpha(\sigma) d \sigma,
\ee
and with $f=(f_e)\in \Lt^2(\Gamma),$ we have $f_e= \textrm{constant}$ for $ e \in \Gamma_{\textrm{sing}};$
\be
\Ht^1(\Gamma) := \bigoplus_{e \in \Gamma_{{\textrm{reg}}}} H^1(e; \alpha, \beta) \bigoplus_{e \in \Gamma_{\textrm{sing}}} \C {\bf 1}_e,
\ee
where for $e \in \Gamma_{\textrm{reg}},$ $H^1(e;\alpha,\beta)$ is the weighted Sobolev space with inner product
\be
(f_e,g_e)_{H^1(e;\alpha,\beta)} = \int_e \left \{\nabla f_e(\sigma) \overline{\nabla g_e(\sigma)} \beta(\sigma) + f_e(\sigma) \overline{g_e(\sigma)} \alpha(\sigma) \right \} d \sigma.
\ee
The inner-products on $\Lt^2(\Gamma)$ and $\Ht^1(\Gamma)$ are respectively,
\be
(f,g) = \sum_{e \in \Gamma} \int_e f(\sigma) \overline{g(\sigma)} \alpha(\sigma) d \sigma \nonumber
\ee
and
\bea
(f,g)_{\Ht^1(\Gamma)} &=& \sum_{e \in \Gamma_{\textrm{reg}}} \int_e \left([\nabla f(\sigma) \cdot \overline{\nabla g(\sigma)}]\beta(\sigma) + [f(\sigma)\overline{g(\sigma)}]\alpha(\sigma) \right) d \sigma \nonumber \\
&+& \sum_{e \in \Gamma_{\textrm{sing}}}\int_e [f(\sigma)\overline{g(\sigma)}]\alpha(\sigma)  d \sigma .
\eea

We define $H_{\Gamma}+ I$ to be the self-adjoint operator in $\Lt^2(\Gamma)$ associated with the $\Ht^1(\Gamma)$ inner-product, where $I$ is the identity operator on $\Lt^2(\Gamma).$ The following theorem is readily proved by a standard argument; cf., \cite{ES}, Theorem 3.3.

\begin{theorem} \label{theorem 5.1}
The domain $D(H_{\Gamma})$ of $H_{\Gamma}$ consists of sequences $u=(u_e)\in \Ht^1(\Gamma)$ which satisfy the following :
\begin{enumerate}
\item $u_e = \textrm{constant}$ for $ e \in \Gamma_{\textrm{sing}};$
\item for each $ e \in \Gamma_{\textrm{reg}}, \beta u_e'$ is locally absolutely continuous on the interior of $e$ and
    \be
    \lim \{ \beta(\sigma) \frac{d u_e}{d\sigma}\} = 0
    \ee
    as $\sigma$ tends to the end points of $e;$
\item for $e \in \Gamma_{\textrm{sing}}, \left(H_{\Gamma}u\right)_e = 0;$
\item for $e \in \Gamma_{\textrm{reg}}
,$
\[
\left(H_{\Gamma}u\right)_e (\sigma) = - \frac{1}{\alpha(\sigma)} \frac{d}{d\sigma} \left[ \beta(\sigma)\frac{du_e}{d\sigma}\right].
\]
\end{enumerate}
\end{theorem}

 An important part in the analysis will be played by the following operator which maps functions on the skeleton to functions on the R\&P  domain. Define  
 \be T_0: \Lt^2(\Gamma) \to L^2(\Omega) \hbox{ by } T_0f=f\circ\tau \hbox{ for } f\in\Lt^2(\Gamma).\ee

 \begin{lemma}\label{adjoint}
 Let $e\in \Gamma.$ Then for $g\in L^2(\tau^{-1}(e))$ and $\textbf{t}(\sigma)\in e$ we have
 \be
 (T_0^*g)(\sigma)=
\frac{1}{\alpha(\sigma)}\int_{-l(\sigma)}^{l(\sigma)} g(\sigma,s)\frac{1}{|\nabla\tau(\sigma,s)|}\ ds. 
 \ee
 \end{lemma}
 
 \begin{proof}
 For $F\in L^2(e)$,
 \begin{eqnarray*} 
\int_e (T_0^*g)(\sigma) \overline{ F(\sigma)} d\sigma  &=&  \int_{\tau^{-1}(e)} g(\sigma) \overline{(T_0 F)(\sigma)} d\sigma \\
 &=& \int_e \overline{F(\sigma)} \left(\frac{1}{\alpha(\sigma)}\int_{-l(\sigma)}^{l(\sigma)} \frac{g(\sigma,s)}{|\nabla\tau(\sigma,s)|}\ ds\right) \alpha(\sigma)\ d\sigma,
  \end{eqnarray*}
proving the result.
 \end{proof}
\begin{remark}\label{adjoint2}
Note that in particular,  for $\textbf{t}(\sigma)\in e\in\Gamma_{\sing}$, the function $(T_0^*g)(\sigma)$ is constant and takes the value
 \be
 (T_0^*g)(\sigma)=
 \frac{1}{|\tau^{-1}(e)|}\int_{\tau^{-1}(e)} g(x)\ dx. 
 \ee
\end{remark}
\begin{lemma}\label{iso}
$T_0$ is an isometry and so $T_0^{*}T_0 = I$, the identity on $\Lt^2(\Gamma)$.
\end{lemma}

\begin{proof}
Let $F\in \Lt^2(\Gamma)$. 
		Then for any edge $e$, we have
		\bea \nonumber
		  \int_{\tau^{-1}(e)}|(F\circ\tau)(x)|^2\ dx
		 &=&  \int_{e}|F(\sigma)|^2\int_{-l(\sigma)}^{l(\sigma)} \frac{1}{|\nabla\tau(\sigma,s)|}\ ds\ d\sigma\\
		 &=&  \int_{e}|F(\sigma)|^2\alpha(\sigma)\ d\sigma.
		\eea
		Adding the contributions from all edges shows that $\norm{T_0F}_{L^2(\Omega)}=\norm{F}_{\widetilde{L}^2(\Gamma)}$ which completes the proof.
\end{proof}

The proof of the next lemma on how $T_0$ interacts
with derivatives is the same as that of \cite[Lemma 3.2]{ES}.

\begin{lemma}
For $F\in \Ht^1(\Gamma)$,
\be
 \sum_{e\in\Gamma} \int_{\tau^{-1}(e)}|\nabla(F\circ\tau)(x)|^2\ dx
		 = \sum_{e\in\Gamma_{\reg}} \int_{e}|F'(\sigma)|^2\beta(\sigma)\ d\sigma.
\ee
\end{lemma}

 
 Let $ \hat H^1(\Omega):= \bigoplus_{e\in \Gamma} H^1 (\tau^{-1}(e))$.
 
 \begin{corollary}\label{T1iso} The map $T_1:  \Ht^1(\Gamma) \to \hat H^1(\Omega)$ given by $T_1f = f\circ \tau$ for $f\in\Ht^1(\Gamma)$
   is an isometry.  
 \end{corollary}

 In the final theorem, $E_1$ denotes the natural embedding of 
 $ \hat H^1(\Omega)$ into the space  $\bigoplus_{e\in \Gamma} L^2   (\tau^{-1}(e))$,
 and $H_\Omega$ denotes the selfadjoint operator associated with   the  $\hat H^1(\Omega)$  inner-product; thus $H_\Omega$ is the orthogonal sum
 $$H_\Omega =  \bigoplus_{e\in \Gamma}  (-\Delta_{N,\tau^{-1}(e)}).
$$

\begin{theorem}\label{Theorem 5.9}
\begin{enumerate}
\item 
 The operator $(H_{\Gamma}+I)^{-1}$ is not compact on $ \Lt^2(\Gamma)$ and $0\in  \sigma_{\textrm{ess}}(H_{\Gamma})$.
 \item $E_1$ is not compact and $0\in \sigma_{\textrm{ess}}(H_\Omega)$.
\item If $E_1(T_1T_1^{*}-I): \hat H^1(\Omega) \rightarrow L^2(\Omega)$ is compact, then $\sigma_{\textrm{ess}}(H_\Omega) \subseteq \sigma_{\textrm{ess}}(H_{\Gamma}).$
\end{enumerate}
\end{theorem}

\begin{proof}
\begin{enumerate}
	\item  On the singular edges $(H_\Gamma+I)^{-1}$ is just $I$. As there are infinitely many singular edges, $(H_\Gamma+I)^{-1}$ is not compact. Moreover, by considering functions supported on one singular edge, we see that $0$ is an eigenvalue of $H_\Gamma$ of infinite multiplicity.
	\item Take an infinite sequence of edges $(e_n)_{n\in\N}$ and consider the sequence of normalised characteristic functions
	$$\varphi_n(\textbf{x}) = \frac{1}{\sqrt{|\tau^{-1}(e_n)|}}\chi_{\tau^{-1}(e_n)}(\textbf{x}) \hbox{ for } \textbf{x}\in\Omega.$$
 This bounded sequence in $\hat H^1(\Omega)$ has no convergent subsequence in $L^2(\Omega)$. Moreover, every $\varphi_n$ is an eigenfunction of $H_\Omega$ with eigenvalue $0$.
	\item The proof of this is similar to that of Theorem 4.4 in \cite{ES}. 
\end{enumerate}
\end{proof}
It is interesting to compare the results we get in Theorem \ref{Theorem 5.9} with those  obtained in \cite{DS} for horn shaped domains  where  the authors  are able to relate  the essential spectrum of the Neumann Laplacian on the horn to that of  a   Schr\"odinger operator on the skeleton.  In our case the re-entrant corners of the Rooms and Passages domain force us to  introduce singular edges which cause the operator on the skeleton to decouple  and prevent such a detailed result from being obtained.


\begin{thebibliography}{22}

\bibitem{Bur} Burenkov, V., Extension of functions preserving certain smoothness and compactness of embeddings for spaces of differentiable functions, \textit{Proc. Steklov Inst. of Math.} {\bf 248} (2005), 1-12.
\bibitem{DS} Davies, E.B. and Simon, B., Spectral properties of the Neumann Laplacian on horns, \textit{Geom. Funct. Anal.} {\bf 2} (1992), 105-117.
\bibitem{EH1} Evans, W.D. and Harris, D.J., Sobolev embeddings for generalized ridged domains, \textit{Proc. London Math. Soc.} {\bf 54} (1987), 141-175.
\bibitem{EH2} Evans, W.D. and Harris, D.J., Fractals, trees and the Neumann Laplacian, \textit{Math. Ann.} {\bf 296} (1993), 493-527.
\bibitem{ES} Evans, W.D. and Saito, Y., Neumann Laplacians on domains and operators on associated trees, \textit{Quart.J.Math.} {\bf 51} (2000), 313--342. 
\bibitem{Fil} Filonov, N., On an inequality for the eigenvalues of the Dirichlet and Neumann problems for the Laplace operator (Russian),
\textit{Algebra i Analiz} {\bf 16} (2004), 172--176. Translation in \textit{St. Petersburg Math. J.} {\bf 16} (2005), 413--416.
\bibitem{HSS} Hempel, R., Seco, L.A. and Simon, B., The essential spectrum of Neumann Laplacians on some bounded singular domains, \textit{Journ. Func. Anal.} {\bf 102} (1991), 448--483.
\bibitem{Hux} Huxley, M.N., Exponential sums and lattice points III, \textit{Proc. Londion Math. Soc.} {\bf 87} (2003), 591--609. 
\bibitem{NS} Netrusov, Y. and Safarov, Y.,  Weyl asymptotic formula for the Laplacian on
domains with rough boundaries, \textit{Commun.Math.Phys.} {\bf 253}
(2005), 481--509.
\end{thebibliography}
\end{document}